\documentclass[preprint,12pt]{elsarticle}

\usepackage{amsfonts}
\usepackage{hyperref}
\usepackage{mathrsfs}
\usepackage{indentfirst}
\usepackage{amsmath}
\usepackage{amssymb}
\usepackage{amsthm}
\usepackage{graphicx}
\usepackage{tikz}
\usepackage{float}
\usepackage{algorithm}
\usepackage{algorithmic}
\usepackage{lipsum}
\usepackage{color}

\newtheorem{definition}{\bf Definition}[section]
\newtheorem{lemma}{\bf Lemma}[section]
\newtheorem{theorem}{\bf Theorem}[section]
\newtheorem{remark}{\bf Remark}[section]
\newtheorem{corollary}{\bf Corollary}[section]
\newtheorem{example}{\bf Example}[section]

\newtheorem{proposition}{\bf Proposition}[section]

\makeatletter
\newenvironment{breakablealgorithm}
  {
   \begin{center}
     \refstepcounter{algorithm}
     \hrule height.8pt depth0pt \kern2pt
     \renewcommand{\caption}[2][\relax]{
       {\raggedright\textbf{\ALG@name~\thealgorithm} ##2\par}%
       \ifx\relax##1\relax 
         \addcontentsline{loa}{algorithm}{\protect\numberline{\thealgorithm}##2}%
       \else 
         \addcontentsline{loa}{algorithm}{\protect\numberline{\thealgorithm}##1}%
       \fi
       \kern2pt\hrule\kern2pt
     }
  }{
     \kern2pt\hrule\relax
   \end{center}
  }
\makeatother

\journal{}

\begin{document}

\begin{frontmatter}



\title{Extremality criteria for the supereigenvector space in max-plus algebra}

\author{Serge{\u{\i}} Sergeev\corref{cors}}

\author{Hui-li Wang\corref{corw}}

\cortext[cors]{University of Birmingham, School of Mathematics, Edgbaston B15 2TT, UK. Email: s.sergeev@bham.ac.uk.}

\cortext[corw]{School of Sciences, Southwest Petroleum University, Chengdu, Sichuan 610500, China. Email:huiliwang77@163.com. Corresponding author.\\
Supported by the National Natural Science Foundation of China (No.11901486), Youth Science and Technology Innovation Team of SWPU for Nonlinear Systems (No. 2017CXTD02).}

\begin{abstract} 
We present necessary and sufficient criteria for a max-algebraic supereigenvector, i.e., a solution of the system $A\otimes\textbf{x}\geq\textbf{x}$ with $A\in\overline{\mathbb{R}}^{n\times n}$ in max-plus algebra, to be an extremal.  We also show that the suggested extremality criteria can be verified in $O(n^2)$ time for any given solution 
$\textbf{x}$.
\end{abstract}

\begin{keyword}
max-plus algebra \sep generator \sep extremal \sep tangent digraph
\MSC 15A80 \sep 15A39
\end{keyword}
\end{frontmatter}

\section{Introduction}\label{intro}

Max-plus algebra, as analogue of linear algebra, defined on $\overline{\mathbb{R}}:=\mathbb{R}\cup \{-\infty\}$ equipped with the operations $(\oplus, \otimes)$, where addition $a\oplus b=\max(a, b)$ and multiplication $a\otimes b=a+b$. Note that $-\infty$ is neutral for $\oplus$ and $0$ is neutral for $\otimes$. The fact that $(\overline{\mathbb{R}}, \oplus, \otimes)$ is a commutative idempotent semiring implies that many tools derived from linear algebra can be applied to max-plus algebra. The advantage of linear techniques enables us to transform nonlinear problems into linear problems with respect to the two operations $(\oplus, \otimes)$ in max-plus algebra. As a result, max-plus algebra has proved to be useful in automata
theory, scheduling theory and discrete event systems, as described in a number of monographs~\cite{Baccelli1992,Peter2010,GondranMinoux2008,MPatWork}.

In max-plus algebra $(\overline{\mathbb{R}}, \oplus, \otimes)$, the operations $(\oplus, \otimes)$ and natural order $\leq$ of real numbers can be extended
to matrices and vectors as in classical linear
algebras. If $A=(a_{ij}),\mbox{}\;B=(b_{ij})$ are matrices with elements from $\overline{\mathbb{R}}$ of
compatible sizes, then we can define the max-plus addition, product and scalar multiples of matrices by $(A\oplus B)_{ij}=a_{ij}\oplus b_{ij}$, $(A\otimes B)_{ij}=\sum_{k}^{\oplus}a_{ik}\otimes
b_{kj}=\mbox{max}_{k}(a_{ik}+b_{kj})$ and $(\lambda\otimes A)_{ij}=\lambda\otimes a_{ij}$ for $\lambda\in \overline{\mathbb{R}}$ respectively.
We write $A\leq B$ if 
$a_{ij}\leq b_{ij}$ for all $i,\mbox{}\;j$.
We denote by $I$ any square matrix whose diagonal entries are $0$ and off-diagonal ones are $-\infty$.
Throughout the paper we also denote by $-\infty$ any vector or matrix whose every component is $-\infty$, by abuse of notation.
For matrices, $-\infty$ is neutral for $\oplus$ and $I$ for $\otimes$. Throughout the paper, the set of indices $\{1,2,\cdots, n\}$ is denoted by $N$.

Given $A\in \overline{\mathbb{R}}^{n\times n}$, the
problem of finding a vector
$\textbf{x}\in\overline{\mathbb{R}}^{n}$, $\textbf{x}\neq-\infty$ and a scalar
$\lambda\in\overline{\mathbb{R}}$ satisfying
\begin{eqnarray}\label{EqSup}
A\otimes\textbf{x}\geq\lambda\otimes\textbf{x}
\end{eqnarray}
is called the
(max-plus algebraic) supereigenproblem, the vector $\textbf{x}$ is called a
supereigenvector of $A$ with associated eigenvalue $\lambda$. The supereigproblem is closely related to the eigenproblem and the subeigenproblem, which are defined by finding a vector
$\textbf{x}\in\overline{\mathbb{R}}^{n}$, $\textbf{x}\neq-\infty$ and a scalar
$\lambda\in\overline{\mathbb{R}}$ satisfying
\begin{eqnarray}\label{EqEig}
A\otimes\textbf{x}=\lambda\otimes\textbf{x}
\end{eqnarray}
and
\begin{eqnarray}\label{EqSub}
A\otimes\textbf{x}\leq\lambda\otimes\textbf{x}
\end{eqnarray}
respectively.  Problems~\eqref{EqEig} and \eqref{EqSub} appeared much earlier and were thoroughly studied in the literature.

Eigenproblem~\eqref{EqEig} plays a vital role in max-plus algebra. The study of this problem first appeared in the work of Cuninghame-Green \cite{R.A.1962} in connection with the study of the steady-state behavior of multi-machine interactive production system, and the existence of an eigenvector was independently proved by Vorobyov \cite{Vorobyov1967}.
Full solution for eigenproblem in the case of irreducible
matrices was given by Cuninghame-Green \cite{R.A.1979} and Gondran and Minoux \cite{M.Gondran77}. 
The case of reducible matrices was treated in detail by Butkovi\v{c} \cite{Peter2010}, but the main ideas can be traced back to 
Gaubert \cite{Gaubert1992} and Bapat et al.~\cite{Bapat.1993} where the spectral theorem for reducible matrices was proved.

The study of max-plus subeigenvectors can be traced back to Gaubert~\cite{Gaubert1992}. However, in the literature on nonnegative matrices they appeared much earlier: see, for example, the work by Engel and Schneider~\cite{Engel1973}. Later, Butkovi\v{c} and Schneider~\cite{Peter2005} used them to describe some useful diagonal similarity scalings and Krivulin~\cite{Krivulin} used them to give a complete solution to certain optimization problems over max-plus algebra.    

Let us start our study of the tropical supereigenvectors with a trivial case. If $\lambda=-\infty$, then $A\otimes\textbf{x}\geq\lambda\otimes\textbf{x}$ for any
$\textbf{x}\in\overline{\mathbb{R}}^{n}$. Hence we only need to consider the case of $\lambda\neq-\infty$ in \eqref{EqSup}. Note that in this case \eqref{EqSup} can equivalently be written as
\begin{eqnarray}\label{EqSupnolam}
A\otimes\textbf{x}\geq\textbf{x}
\end{eqnarray}
by multiplying both sides of \eqref{EqSup} by $\lambda^{-1}$ 
and then writing $A$ instead of $\lambda^{-1}\otimes A$.
Then we denote $\mathscr{X}=\{\textbf{x}\in\overline{\mathbb{R}}^{n}:
A\otimes\textbf{x}\geq\textbf{x}, \textbf{x}\neq-\infty\}$ and $\mathscr{X}'=\mathscr{X}\cup\{-\infty\}$ in the case of $\mathscr{X}\neq\emptyset$. 
Set $\mathscr{X}'$ is the solution space of \eqref{EqSupnolam} with $A\in\overline{\mathbb{R}}^{n\times n}$. System~\eqref{EqSupnolam} consists of 
a finite number of inequalities and hence, by the result of~\cite{Peter1984}, it is a finitely generated subspace of $\overline{\mathbb{R}}^{n}$. For $A\otimes \textbf{x}\geq \textbf{x}$ with $A\in \overline{\mathbb{R}}^{n\times n}$ and $\textbf{x}=(x_1, x_2, \cdots, x_n)^T\in\overline{\mathbb{R}}^{n}$, $A_{i}\otimes\textbf{x}\geq x_{i}$ stands for the $i$-inequality of the system $A\otimes \textbf{x}\geq \textbf{x}$, where $A_{i}$ denotes the $i$th row of $A$.

The problem of finding all max-plus supereigenvectors of a given matrix was posed by Butkovi\v{c}, Schneider and Sergeev \cite{Peter2012}, motivated by the study of robustness in max-plus linear systems. An algorithm which can be used to find all generators for the solution space of \eqref{EqSupnolam} was subsequently presented by Wang and Wang \cite{wanghuili2014}, who also gave two algorithms for finding a finite proper supereigenvector for irreducible matrix and a finite solution of the system  \eqref{EqSupnolam} for reducible matrix \cite{wanghuili2015InformationSciences}. Based on the results in \cite{wanghuili2014}, Wang, Yang and Wang developed an improved algorithm which can be used to find a smaller set of generators \cite{wanghuili2020}, and Sergeev~\cite{Sergeev2015}  gave a combinatorial description of those generators obtained by Wang and Wang~\cite{wanghuili2014} that are extremal.
Note that another method for finding a non-trivial subset of a set of generators of finite supereigenvectors follows from the proof of the main statement in the work of Butkovi\v{c}~\cite{Peter2016}.

Note that \eqref{EqSupnolam} can be seen as a two-sided system $A\otimes\textbf{x}\geq B\otimes\textbf{x}$, where $B$ is a max-plus identity matrix.  Checking the solvability 
$A\otimes\textbf{x}\geq B\otimes\textbf{x}$ is a known NP$\cap$Co-NP problem \cite{Bezem2010}, for which no polynomial method has been found. Recently Allamigeon, Gaubert and Goubault \cite{Allamigeon2013} suggested a major improvement of the method of Butkovi\v{c}~\cite{Peter1984} to produce a set of generators for the solution space of such problem and also gave a criterion of extremality for a given solution in \cite[Theorem 1]{Allamigeon2013}. Their criterion is based on the notion of tangent directed hypergraph and connectivity in a hypergraph, and its verification requires an almost linear time. More precisely, when applied to \eqref{EqSupnolam}  with $A\in \overline{\mathbb{R}}^{n\times n}$, it requires $O(n^2)\times \alpha(n)$ time, where 
$\alpha(n)$ is the inverse of the Ackermann function.

One of the main ideas of the present paper is that for \eqref{EqSupnolam}, the tangent directed hypergraph can be replaced with a digraph (which we call tangent), and we do not need the more elaborate notion of connectivity in hypergraphs. The main result of the present paper is Theorem~\ref{Th:main}. It states three criteria which an extremal generator should satisfy, and one of these criteria means that the tangent digraph should be intercyclic.  We show that the joint verification of these criteria requires $O(n^2)$ time, i.e., it is linear (although verifying the intercyclic property alone is more complex~\cite{Mccuaig1993}). The joint verification of these criteria can be performed by a simple procedure, which we describe in the paper. However, such reduction in complexity is achieved only for the system $A\otimes\textbf{x}\geq \textbf{x}$ and not in the general case.


The rest of the paper is organized as follows. In Section~\ref{s:prel}, we give a brief overview of basic facts in max-plus algebra which we will need, following \cite{Peter2010}, and introduce the notion of tangent digraph. In Section~\ref{s:extre} we formulate and prove the main result of this paper, Theorem~\ref{Th:main}. This section is split into two subsections: the first subsection introduces and studies the notions of variable and invariable nodes of a tangent digraph, and the second subsection contains the main part of the proof of Theorem~\ref{Th:main}. Finally, in Section~\ref{s:comp} we show that the joint verification of the criteria stated in Theorem~\ref{Th:main} takes $O(n^2)$ time.


\section{Preliminaries}
\label{s:prel}

Given a square matrix $A=(a_{ij})\in \overline{\mathbb{R}}^{n\times n}$ the symbol
$D_{A}$ denotes the weighted directed graph $(N, E, w)$ with the node set $N$ and arc set
$E=\{(i, j): a_{ij}\neq-\infty\}$ and weight
$\omega(i,j) = a_{ij}$ for all $(i,j)\in E$.

A sequence of nodes $\pi=(i_{1},\cdots, i_{k})$ in $D_{A}$ is called a {\em path} (in $D_{A}$) if $(i_{j}, i_{j+1})\in E$ for all $j=1,\cdots,k-1$.
The length of such path $\pi$, denoted by $l(\pi)$, is $k-1$. The weight of such path $\pi$ is denoted by $\omega(\pi)$ and is equal to $a_{i_{1}i_{2}}+\cdots+a_{i_{k-1}i_{k}}$. If $\omega(\pi)\geq 0$ then $\pi$ is called a {\em nonnegative path.} Node $i_{1}$ is called the {\em starting node} and $i_{k}$ the {\em endnode} of $\pi$, respectively. A path from $i_{1}$ to $i_{k}$ can be called a $i_{1}-i_{k}$ path. Further $i_{k}$ is said to be {\em reachable} or {\em accessible} from $i_{1}$, notation $i_{1}\rightarrow i_{k}$.
If $i_{1}=i_{k}$ then the path $(i_{1},\cdots, i_{k})$ is called a {\em cycle}. Moreover, if $i_{p}\neq i_{q}$ for all $p, q=1,\cdots, k-1,
p\neq q$ then that cycle is called an {\em elementary cycle} (similarly, we can define {\em elementary path}). Cycle $(i_{1}, i_{2}=i_{1})$ is called a {\em loop}. 

Let $S\subseteq\overline{\mathbb{R}}^{n}$. We say that $S$ is a {\em subspace} of $\mathbb{R}^n$ if $a\otimes\textbf{u}\oplus b\otimes\textbf{v}\in S$ for every $\textbf{u}, \textbf{v}\in S$ and $a, b\in \overline{\mathbb{R}}$.
A vector $\textbf{v}=(v_{1}, \cdots,
v_{n})^{T}\in\overline{\mathbb{R}}^{n}$ is called a {\em max-combination}
of $S$ if $\textbf{v}=\sum^{\oplus}_{\textbf{x}\in
S}\alpha_{\textbf{x}}\otimes \textbf{x}, \alpha_{\textbf{x}}\in
\overline{\mathbb{R}}$, where only a finite number of
$\alpha_{\textbf{x}}$ are real. Denote the set of all max-combinations of
$S$ by $\operatorname{span}(S)$. $S$ is called a set of {\em generators}
for $T$ if $\operatorname{span}(S)=T$. The set $S$ is called {\em (weakly) dependent} if $\textbf{v}$ is a
max-combination of $S\backslash\{\textbf{v}\}$ for some $\textbf{v}\in S$.
Otherwise $S$ is (weakly) independent. Let $S,
T\subseteq\overline{\mathbb{R}}^{n}$. If $S$ is an independent set of generators for $T$ then the set $S$ is called a basis
of $T$.

A vector $\textbf{v}\in S$ is called an {\em extremal} in $S$ if for all $\textbf{x}, \textbf{y}\in S$, $\textbf{v}=\textbf{x}\oplus \textbf{y}$ implies $\textbf{v}=\textbf{x}$ or $\textbf{v}=\textbf{y}$. A vector $\textbf{v}=(v_{1}, \cdots, v_{n})^{T}\in S$ ($\textbf{v}\neq -\infty$)is called {\em scaled} if $\|\textbf{v}\|=0$, where $\|\textbf{v}\|=\max_{i=1}^{n}v_{i}$.
The set $S$ is called scaled if all its elements are scaled. For vector $\textbf{v}=(v_{1}, \cdots, v_{n})^{T}\in \overline{\mathbb{R}}^{n}$ the support of $\textbf{v}$ is defined by
$$\text{Supp}(\textbf{v})=\{j\in N: v_{j}\in \mathbb{R}\}.$$

\if{
Given $A\in \overline{\mathbb{R}}^{n\times n}$, the symbol
$\lambda(A)$ will stand for the maximum cycle mean of $A$, that is,
\begin{eqnarray}\label{Eq2.3}
\lambda(A)=\mathop{\mbox{max}}_{\sigma}\mu(\sigma, A),
\end{eqnarray}
where the maximization is taken over all elementary cycles in
$D_{A}$, and
\begin{eqnarray}\label{Eq2.4}
\mu(\sigma, A)=\frac{\omega(\sigma, A)}{l(\sigma)}
\end{eqnarray}
denotes the mean of a cycle $\sigma$. Clearly, $\lambda(A)$ always
exists since the number of elementary cycles is finite.
}\fi

The following result explains the importance of extremals for subspaces of $\overline{\mathbb{R}}^n$.

\begin{proposition}[\bf Wagneur~\cite{Wagneur1991}, see also \cite{Peter2010}]
\label{Pro0.1}
If $T$ is a finitely generated subspace then its set of scaled extremals is nonempty and it is the unique scaled basis for $T$.
\end{proposition}

By Proposition \ref{Pro0.1}, the set of scaled extremals in $\mathscr{X}'$ is the unique scaled basis for $\mathscr{X}'$. Therefore it is important to judge whether a solution (or a generator) in the solution space of the system (\ref{EqSupnolam}) is an extremal 
in $\mathscr{X}'$. 

Using the definition of extremal we readily obtain:

\begin{proposition}
\label{Th1.1}
Let $\emph{\textbf{x}}=(x_1, x_2, \cdots, x_n)^T\in\mathscr{X}$. If $|\emph{\text{Supp}}(\emph{\textbf{x}})|=1$, then $\emph{\textbf{x}}$ is an extremal in $\mathscr{X}'$.
\end{proposition}

In what follows we only need to discuss the extremality of a solution $\textbf{x}$ with $|\text{Supp}(\textbf{x})|\geq 2$.

For any $\textbf{x}=(x_1, x_2, \cdots, x_n)^T\in\mathscr{X}$, take any $i\in \text{Supp}(\textbf{x})$ and denote
\begin{align*}
&V_{\textbf{x}}(i)=\{j\in N: a_{ij}+x_{j}=\max\limits_{k\in N}(a_{ik}+x_{k})=x_{i}\}.
\end{align*}

The following notion can be seen as very closely related to the tangent directed hypergraph introduced in \cite{Allamigeon2013},
and therefore we use a similar terminology here.

\begin{definition}
The {\rm tangent digraph} $\mathcal{D}_{\emph{\textbf{x}}}=(\text{Supp}(\emph{\textbf{x}}), E_{\emph{\textbf{x}}})$ at
$\emph{\textbf{x}}$ is defined as the digraph with the node set $\text{Supp}(\emph{\textbf{x})}$ and arc set
$E_{\emph{\textbf{x}}}=\{(j, i): i\in \text{Supp}(\emph{\textbf{x}}) \mbox{ and } j\in V_{\emph{\textbf{x}}}(i)\}$.
\end{definition}

In $\mathcal{D}_{\textbf{x}}$, the arc $(j, i)$ is called an incoming arc of node $i$ from node $j$ or outgoing arc of node $j$ to node $i$. The node $j$ is called the starting node and $i$ the endnode of $(j, i)\in E_{\textbf{x}}$, respectively. Specially if $(i, i)\in E_{\textbf{x}}$ for node $i\in \text{Supp}(\textbf{x})$, then arc $(i, i)$ is both incoming arc and outgoing arc of node $i$. 

\section{Extremality test}
\label{s:extre}

Theorem 1 of Allamigeon et al.~\cite{Allamigeon2013} implies that $\textbf{x}$ is an extremal if and only if there is a strongly connected component in the tangent directed hypergraph at $\textbf{x}$ which can be accessed from any node in $\text{Supp}(\textbf{x})$ via a certain hypergraph access relation. We will be aiming at alternative criteria to characterize extremality, which will not need hypergraphs and access relations on them.  Although the criteria which we are going to obtain are equivalent to \cite[Theorem 1]{Allamigeon2013}, we will prefer to obtain them in an elementary way, rather than deducing them from the result of \cite{Allamigeon2013}.

\subsection{Variable nodes and extremality}

The following definitions, given in terms of the tangent digraph  $\mathcal{D}_{\textbf{x}}$ at $\textbf{x}$, where
$\textbf{x}=(x_1, x_2, \cdots, x_n)^T\in\mathscr{X}$, will be key to our characterization.

\begin{definition}
\label{De1.1}
Let $\emph{\textbf{x}}=(x_1, x_2, \cdots, x_n)^T\in\mathscr{X}$ and $i\in \emph{\text{Supp}}(\emph{\textbf{x}})$. If there is no outgoing arc other than possibly $(i, i)$  in $\mathcal{D}_{\emph{\textbf{x}}}$ for node $i$, then $i$ is called a {\rm \uppercase\expandafter{\romannumeral1}-variable node.}
\end{definition}

\begin{definition}
\label{De1.2}
Let $\emph{\textbf{x}}=(x_1, x_2, \cdots, x_n)^T\in\mathscr{X}$ and $i\in \emph{\text{Supp}}(\emph{\textbf{x}})$. If the endnode of any outgoing arc from $i$ except $(i, i)$ (if it exists) has another incoming arc in $\mathcal{D}_{\emph{\textbf{x}}}$, then $i$ is called a {\rm \uppercase\expandafter{\romannumeral2}-variable node.}
\end{definition}

\begin{definition}
\label{De1.3}
Let $\emph{\textbf{x}}=(x_1, x_2, \cdots, x_n)^T\in\mathscr{X}$ and $i\in \emph{\text{Supp}}(\emph{\textbf{x}})$. If there exists a node $j$, $j\neq i$ such that $(i, j)\in E_{\emph{\textbf{x}}}$ and $j$ has no other incoming arc in $\mathcal{D}_{\emph{\textbf{x}}}$, then $i$ is called an {\rm invariable node.}
\end{definition}

A node will be called {\em variable} if it is \uppercase\expandafter{\romannumeral1}-variable or \uppercase\expandafter{\romannumeral2}-variable. The proof of the following statement is straightforward.

\begin{lemma}
\label{Lem1}
Let $\emph{\textbf{x}}=(x_1, x_2, \cdots, x_n)^T\in\mathscr{X}$. Then each node in $\emph{\text{Supp}}(\emph{\textbf{x}})$ is either a \uppercase\expandafter{\romannumeral1}-variable node or a \uppercase\expandafter{\romannumeral2}-variable node or an invariable node.
\end{lemma}


Let $\textbf{x}=(x_1, x_2, \cdots, x_n)^T\in\mathscr{X}$. The following proposition enables us to construct $\textbf{x}'\in\mathscr{X}$ such that $\textbf{x}'\leq\textbf{x}$ and $\textbf{x}'\neq\textbf{x}$.

\begin{proposition}
\label{Pro1.1}
Let $\emph{\textbf{x}}=(x_1, x_2, \cdots, x_n)^T\in\mathscr{X}$ and $i\in \emph{\text{Supp}}(\emph{\textbf{x}})$ be a variable node. Then there exists $\emph{\textbf{x}}'=(x_1', x_2', \cdots, x_n')$ with $x_{i}'<x_{i}$ and $x_{j}'=x_{j}$ for any $j\neq i, j\in N$ such that $\emph{\textbf{x}}'\in\mathscr{X}$.
\end{proposition}
\begin{proof}
Denote
\begin{equation*}
\begin{split}
J=\{j\in \text{Supp}(\textbf{x}) \text{ and } j\neq i:\; & a_{ji}+x_{i}>x_{j}, a_{js}+x_{s}<x_{j}\\ 
                                               &\text{ for any } s\neq i, s\in \text{Supp}(\textbf{x})\}.
\end{split}
\end{equation*}
Construct a vector $\textbf{x}'=(x_1', x_2', \cdots, x_n')$ with $x_{j}'=x_{j}$ for any $j\neq i, j\in N$ and $x_{i}'$ satisfying $x_{i}'<x_{i}$ when $J=\emptyset$ or $\max\limits_{j\in J}\{x_{j}-a_{ji}\}\leq x_{i}'<x_{i}$ when $J\neq\emptyset$. We claim that $\textbf{x}'\in\mathscr{X}$, that is, every inequality of the system $A\otimes\textbf{x}\geq\textbf{x}$ holds for $\textbf{x}'$.



Let us first consider the $i$-th inequality. Note that $x_{i}'<x_{i}$ both when $J=\emptyset$ and when $J\neq\emptyset$.
There are two possible cases to consider:
\begin{itemize}
\item[(1)] If $a_{ii}<0$, then there exists $s\neq i, s\in \text{Supp}(\textbf{x})$ such that $a_{is}+x_{s}\geq x_{i}$ since $\textbf{x}\in\mathscr{X}$. Then $A_{i}\otimes\textbf{x}'\geq a_{is}+x_{s}'=a_{is}+x_{s}\geq x_{i}> x_{i}'$, hence
$A_{i}\otimes\textbf{x}'\geq x_{i}'$, as required. 
\item[(2)] If $a_{ii}\geq 0$, then obviously $A_{i}\otimes\textbf{x}'\geq a_{ii}+x_{i}'\geq x_{i}'$. 
\end{itemize}

Next, we shall discuss the $j$-th inequality for all $j\neq i$ and $j\in N$.
We will have to consider the following cases:\\
\begin{itemize}
\item[(1)] If $j\notin \text{Supp}(\textbf{x})$, then obviously $A_j\otimes \textbf{x}'\geq x_{j}'=-\infty$.
\item[(2)] If $j\in \text{Supp}(\textbf{x})$ and $a_{ji}+x_{i}<x_{j}$ then there exists $s\neq i, s\in \text{Supp}(\textbf{x})$ such that $a_{js}+x_{s}\geq x_{j}$ since $\textbf{x}\in\mathscr{X}$. Then $A_{j}\otimes\textbf{x}'\geq a_{js}+x_{s}'=a_{js}+x_{s}\geq x_{j}=x_{j}'$, thus  $A_{j}\otimes\textbf{x}'\geq x_{j}'$.
\item[(3)] If $j\in \text{Supp}(\textbf{x})$, $a_{ji}+x_{i}>x_{j}$ and $j\notin J$, then $a_{js}+x_{s}\geq x_{j}$ for some $s\neq i, s\in \text{Supp}(\textbf{x})$. Then we obtain $A_{j}\otimes\textbf{x}'\geq a_{js}+x_{s}'=a_{js}+x_{s}\geq x_{j}=x_{j}'$, thus
$A_j\otimes\textbf{x}'\geq x_{j}'$.
\item[(4)] If $j\in\text{Supp}(\textbf{x})$, $a_{ji}+x_{i}>x_{j}$ and $j\in J$ , then $a_{js}+x_{s}<x_{j}$ for any $s\neq i, s\in \text{Supp}(\textbf{x})$. Then we obtain $A_{j}\otimes\textbf{x}'\geq a_{ji}+x_{i}'\geq a_{ji}+\max\limits_{j\in J}\{x_{j}-a_{ji}\}\geq a_{ji}+x_{j}-a_{ji}=x_{j}=x_{j}'$, thus again  $A_j\otimes \textbf{x}'\geq x'_j$.
\item[(5)] If $j\in \text{Supp}(\textbf{x})$, $a_{ji}+x_{i}=x_{j}$ and  $(i,j)\notin E_{\textbf{x}}$, then there exists $s\neq i, s\in\text{Supp}(\textbf{x})$ such that  $a_{js}+x_{s}>x_{j}$,  so we obtain $A_{j}\otimes\textbf{x}'\geq a_{js}+x_{s}'=a_{js}+x_{s}> x_{j}=x_{j}'$,
thus $A_j\otimes\textbf{x}'\geq x_{j}'$.
\item[(6)] If $j\in \text{Supp}(\textbf{x})$, $a_{ji}+x_{i}=x_{j}$ and $(i,j)\in E_{\textbf{x}},$ then $i$ is a II-variable node, which implies that there is $s\neq i$ with $(s,j)\in E_{\textbf{x}}$, meaning that $a_{js}+x_s=x_j$ and $a_{jk}+x_k\leq x_j$ for all $k\neq i,s$.
Then we have  $A_{j}\otimes\textbf{x}'\geq a_{js}+x_{s}'=a_{js}+x_{s}=x_{j}=x_{j}'$, thus $A_{j}\otimes\textbf{x}'\geq x'_j$.
\end{itemize}


\end{proof}

Denote $\mathscr{X}^{\leq\textbf{x}}=\{\textbf{x}'\in\mathscr{X}: \textbf{x}'\leq\textbf{x}, \textbf{x}'\neq \textbf{x}\}$. If $\textbf{x}$ has at least one finite component, then $\mathscr{X}^{\leq\textbf{x}}\neq\emptyset$, since $a\otimes\textbf{x}\in\mathscr{X}^{\leq\textbf{x}}$ for any $a\in \mathbb{R}$ and $a<0$. Then the following two simple but useful lemmas can be straightforwardly deduced from the definition of extremal.

\begin{lemma}
\label{Lem2}
Let $\emph{\textbf{x}}=(x_1, x_2, \cdots, x_n)^T\in\mathscr{X}$. If $\emph{\textbf{x}}\neq\emph{\textbf{x}}^{1}\oplus\emph{\textbf{x}}^{2}$ for any $\emph{\textbf{x}}^{1}, \emph{\textbf{x}}^{2}\in\mathscr{X}^{\leq\emph{\textbf{x}}}$, then
$\emph{\textbf{x}}$ is an extremal in $\mathscr{X}'$.
\end{lemma}

\begin{lemma}
\label{Lem3}
Let $\emph{\textbf{x}}=(x_1, x_2, \cdots, x_n)^T\in\mathscr{X}$. If $\emph{\textbf{x}}=\emph{\textbf{x}}^{1}\oplus\emph{\textbf{x}}^{2}$ for some $\emph{\textbf{x}}^{1}, \emph{\textbf{x}}^{2}\in\mathscr{X}^{\leq\emph{\textbf{x}}}$, then
$\emph{\textbf{x}}$ is not an extremal in $\mathscr{X}'$.
\end{lemma}

Then by Proposition \ref{Pro1.1} we can get a useful statement:

\begin{corollary}
\label{Cor:cond3}
Let $\emph{\textbf{x}}=(x_1, x_2, \cdots, x_n)^T\in\mathscr{X}$. If there are at least two variable nodes in $\mathcal{D}_{\emph{\textbf{x}}}$, then $\emph{\textbf{x}}$ is not an extremal.
\end{corollary}
\begin{proof}
Without loss of generality assume that there are two variable nodes in $\mathcal{D}_{\textbf{x}}$, say $i, j$, then by Proposition \ref{Pro1.1} we can get two different solutions $\textbf{x}^{1}, \textbf{x}^{2}\in\mathscr{X}^{\leq\textbf{x}}$, where $\textbf{x}^{1}=(x^{1}_1, x^{1}_2, \cdots, x^{1}_n)^T$ with $x^{1}_i<x_{i}$ and $x^{1}_{s}=x_{s}$ for any $s\neq i, s\in N$, $\textbf{x}^{2}=(x^{2}_1, x^{2}_2, \cdots, x^{2}_n)^T$ with $x^{2}_{j}<x_{j}$ and $x^{2}_{t}=x_{t}$ for any $t\neq j, t\in N$. Obviously $\textbf{x}=\textbf{x}^{1}\oplus\textbf{x}^{2}$. Hence by Lemma \ref{Lem3} $\textbf{x}$ is not an extremal.
\end{proof}

In contrast to Proposition \ref{Pro1.1}, we have:

\begin{proposition}
\label{Pro1.2}
Let $\emph{\textbf{x}}=(x_1, x_2, \cdots, x_n)^T\in\mathscr{X}$ and $i\in \emph{\text{Supp}}(\emph{\textbf{x}})$ be an invariable node. Then $\emph{\textbf{x}}'\notin\mathscr{X}$ for any $\emph{\textbf{x}}'=(x_1', x_2', \cdots, x_n')$ satisfying $x_{i}'<x_{i}$ and $x_{j}'=x_{j}$ for all $j\neq i, j\in N$.
\end{proposition}
\begin{proof}
Let $i\in \text{Supp}(\textbf{x})$ be an invariable node. By the definition of invariable node, there exists node $t\in \text{Supp}(\textbf{x})$, $t\neq i$ such that $(i, t)\in E_{\textbf{x}}$ and $t$ has no other incoming arc, that is, $a_{ti}+x_{i}=x_{t}$ and $a_{tj}+x_{j}<x_{t}$ for all $j\in \text{Supp}(\textbf{x})$, $j\neq i$. Then for any $\textbf{x}'=(x_1', x_2', \cdots, x_n')$ with $x_{i}'<x_{i}$ and $x_{j}'=x_{j}$ for all $j\neq i, j\in N$ we obtain
$$A_{t}\otimes\textbf{x}'=\sum\limits_{k\in N}^{\oplus}(a_{tk}+x_{k}')=(a_{ti}+x_{i}')\oplus\sum\limits_{j\neq i}^{\oplus}(a_{tj}+x_{j})<x_{t}.$$
Thus, $\textbf{x}'\notin\mathscr{X}$.
\end{proof}

\subsection{Extremality criteria}

The goal of this subsection are the extremality criteria stated in Theorem~\ref{Th:main}. Note that in Corollary~\ref{Cor:cond3} we already established that the third of these
criteria is necessary, and Propositions~\ref{Pro:Cond1} and~\ref{Pro:Cond2} will serve to prove that the first two criteria are also necessary. Corollary~\ref{c:important} will be another important result treating the case of no variable nodes. After this we will focus on the proof of sufficiency of the criteria stated in Theorem~\ref{Th:main}.

The next definition will be required to formulate one of the criteria of Theorem~\ref{Th:main}.

\begin{definition}
\label{De1.4}
Let $\emph{\textbf{x}}=(x_1, x_2, \cdots, x_n)^T\in\mathscr{X}$ and $W\subseteq \emph{\text{Supp}}(\emph{\textbf{x}})$.
If, for any node of $W$, neither ingoing arcs from nodes of $\emph{\text{Supp}}(\emph{\textbf{x}})\backslash W$ to that node of $W$ nor outgoing arcs from that node of $W$ to nodes of $\emph{\text{Supp}}(\emph{\textbf{x}})\backslash W$ exist in $\mathcal{D}_{\emph{\textbf{x}}}$, then $W$ is called an {\rm isolated node set} in $\mathcal{D}_{\emph{\textbf{x}}}$.
\end{definition}


Note that if $W\subsetneqq \text{Supp}(\textbf{x})$ is an isolated node set, then there are at least two isolated node sets in $\text{Supp}(\textbf{x})$: they are $W$ and $\text{Supp}(\textbf{x})\backslash W$.

\begin{proposition}
\label{Pro:Cond1}
Let $\emph{\textbf{x}}=(x_1, x_2, \cdots, x_n)^T\in\mathscr{X}$. If there exists a proper subset of $\emph{\text{Supp}}(\emph{\textbf{x}})$ that is isolated in $\mathcal{D}_{\emph{\textbf{x}}}$, then $\emph{\textbf{x}}$ is not an extremal.
\end{proposition}
\begin{proof}
Let $W_1\subsetneqq \text{Supp}(\textbf{x})$ is an isolated node set in $\mathcal{D}_{\textbf{x}}$. Then the complement
$W_2=\text{Supp}(\textbf{x})\backslash W_1$ is also isolated in $\mathcal{D}_{\textbf{x}}$.
Let $W_1=\{i_{1}, i_{2},\cdots, i_{s}\}$ and $W_2=\{j_{1}, j_{2},\cdots, j_{t}\}$.
We claim that there exist two different solutions $\textbf{x}^{1}, \textbf{x}^{2}\in\mathscr{X}^{\leq\textbf{x}}$ such that $\textbf{x}=\textbf{x}^{1}\oplus\textbf{x}^{2}$, where $\textbf{x}^{1}=(x^{1}_1, x^{1}_2, \cdots, x^{1}_n)^T$ with $x^{1}_i<x_{i}$ for all $i\in W_{1}$ and $x^{1}_j=x_{j}$ for all $j\notin W_1$ and $\textbf{x}^{2}=(x^{2}_1, x^{2}_2, \cdots, x^{2}_n)^T$ with $x^{2}_j<x_{j}$ for all $j\in W_2$ and $x^{2}_i=x_{i}$ for all $i\notin W_2$.

By symmetry, it suffices to prove the existence of $\textbf{x}^{1}$ with the above properties.  Denote
$$K=\{i\in W_{2}: a_{ij}+x_{j}>x_{i} \text{ for some } j\in W_{1} \text{ and } a_{ii'}+x_{i'}<x_{i} \text{ for all } i'\in W_{2} \}$$
If $K\neq \emptyset$, then for any $i\in K$ denote $I_{i}=\{j\in W_{1}: a_{ij}+x_{j}>x_{i}\}$ and denote
$$a=\min\limits_{i\in K}\max\limits_{j\in I_{i}}(a_{ij}+x_{j}-x_{i}).$$

Note also that, since $A_i\otimes\textbf{x}\geq x_i$ and since $W_1$ and $W_2$ are isolated in  $\mathcal{D}_{\textbf{x}}$, we can simplify $K$ as follows:
$$K=\{i\in W_{2}\colon a_{ij}+x_{j}<x_{i} \text{ for all } j\in W_{2} \}$$

We will set $\textbf{x}^{1}=(x^{1}_1, x^{1}_2, \cdots, x^{1}_n)^T$ with $x^{1}_i=x_{i}-b$ for any $i\in W_{1}$ and $x^{1}_j=x_{j}$ for any $j\notin W_{1}$. Here $b$ is a real number that satisfies $0<b\leq a$ when $K\neq\emptyset$, and arbitrary positive real number when $K=\emptyset$.  Obviously such $\textbf{x}^{1}$ satisfies $\textbf{x}^{1}\leq\textbf{x}$ and $\textbf{x}^{1}\neq\textbf{x}$, so we only need to prove $A\otimes\textbf{x}^{1}\geq\textbf{x}^{1}$. We need to consider the following special cases:
\begin{itemize}
\item[(1)] If $i\notin\text{Supp}(\textbf{x})$, then the inequality $A_i\otimes \textbf{x}^{1}\geq x^1_i=-\infty$ is obvious.
\item[(2)] If $i\in W_{1}$ and $(j, i)\in E_{\textbf{x}}$ for some $j$, then $j\in W_1$ since $W_1$ is isolated, and
$A_{i}\otimes\textbf{x}^{1}\geq a_{ij}+x_{j}^{1}=a_{ij}+x_{j}-b=x_{i}-b=x_{i}^{1}$, thus $A_{i}\otimes\textbf{x}^{1}\geq x_i^1$.
\item[(3)] If $i\in W_{1}$ and there exists $j\in W_{1}$ such that $a_{ij}+x_{j}>x_{i}$, then $A_{i}\otimes\textbf{x}^{1}\geq a_{ij}+x_{j}^{1}=a_{ij}+x_{j}-b>x_{i}-b=x_{i}^{1}$, thus $A_{i}\otimes\textbf{x}^{1}\geq x_i^1$.
\item[(4)] If $i\in W_{1}$ and there exists $j\notin W_{1}$ such that $a_{ij}+x_{j}>x_{i}$, then  $A_{i}\otimes\textbf{x}^{1}\geq a_{ij}+x_{j}^{1}=a_{ij}+x_{j}>x_{i}>x_{i}-b=x_{i}^{1}$, thus  $A_{i}\otimes\textbf{x}^{1}\geq x_i^1$.%
\item[(5)] If $i\in W_{2}$ and $i\notin K$, then there exists $j\in W_{2}$ such that $a_{ij}+x_{j}\geq x_{i}$.
Then $A_{i}\otimes\textbf{x}^{1}\geq a_{ij}+x_{j}^{1}=a_{ij}+x_{j}\geq x_{i}=x_{i}^{1}$, thus  $A_{i}\otimes\textbf{x}^{1}\geq x_i^1$.
\item[(6)] If $i\in W_{2}$ and $i\in K$, then taking into account that $K\neq\emptyset$ and using that $0<b\leq a$ and
the definition of $a$ given earlier, we obtain:
\begin{align*}
&A_{i}\otimes\textbf{x}^{1}\geq\sum\limits_{j\in I_{i}}^{\oplus}(a_{ij}+x^{1}_j)\geq\sum\limits_{j\in I_{i}}^{\oplus}(a_{ij}+x_j-a)\\
&=\sum\limits_{j\in I_{i}}^{\oplus}(a_{ij}+x_j)-a\geq\sum\limits_{j\in I_{i}}^{\oplus}(a_{ij}+x_j)-\max\limits_{j\in I_{i}}(a_{ij}+x_{j}-x_{i})\\
&=\sum\limits_{j\in I_{i}}^{\oplus}(a_{ij}+x_j)-\sum\limits_{j\in I_{i}}^{\oplus}(a_{ij}+x_{j})+x_{i}\\
&=x_{i}=x_{i}^{1},
\end{align*}
thus  $A_{i}\otimes\textbf{x}^{1}\geq x_i^1$.
\end{itemize}
Solution $\textbf{x}^2$ can be constructed similarly, and solutions $\textbf{x}^{1}$ and $\textbf{x}^{2}$ constructed by the above method obviously satisfy $\textbf{x}=\textbf{x}^{1}\oplus\textbf{x}^{2}$, where $\textbf{x}^{1}, \textbf{x}^{2}\in\mathscr{X}^{\leq\textbf{x}}$, so $\textbf{x}$ is not an extremal.
\end{proof}

Let $\textbf{x}=(x_1, x_2, \cdots, x_n)^T\in\mathscr{X}$. Denote $L_{\sigma}=\{i\in \text{Supp}(\textbf{x}): i\in\sigma\}$ for some cycle $\sigma=(i_{1}, i_{2}, \cdots, i_{s}, i_{s+1}=i_{1})$ in $\mathcal{D}_{\textbf{x}}$ and denote the set of all nodes in an arbitrarily taken path $i-j$ in $\mathcal{D}_{\textbf{x}}$ by $L_{i-j}$. Specifically, the set of all nodes in path $(j_{1}, j_{2}, \cdots, j_{t})$ in $\mathcal{D}_{\textbf{x}}$ can be denoted by $L_{(j_{1}, j_{2}, \cdots, j_{t})}$.

\begin{proposition}
\label{Pro:Cond2}
Let $\emph{\textbf{x}}=(x_1, x_2, \cdots, x_n)^T\in\mathscr{X}$. 
If there are at least two cycles in $\mathcal{D}_{\emph{\textbf{x}}}$ whose node sets are disjoint, then $\emph{\textbf{x}}$ is not an extremal.
\end{proposition}
\begin{proof}
Let $\sigma_1$ and $\sigma_2$ be two node disjoint cycles in $\mathcal{D}_{\textbf{x}}$. Denote
\begin{align*}
\Sigma_{1}=&\; L_{\sigma_{1}}\cup\{s\in \text{Supp}(\textbf{x})\backslash (L_{\sigma_{1}}\cup L_{\sigma_{2}})\colon \\
& (\exists i\in L_{\sigma_{1}})(i\rightarrow s \mbox{ }\text{and} \mbox{ } L_{i-s}\cap L_{\sigma_{2}}=\emptyset
\mbox{ }\text{for some}\mbox{ } i-s\mbox{ }\text{path})\},\\
\Sigma_{2}=&\; L_{\sigma_{2}}\cup\{t\in \text{Supp}(\textbf{x})\backslash (L_{\sigma_{1}}\cup L_{\sigma_{2}})\mbox{ } \text{and}\mbox{ }t\notin\Sigma_{1}\colon\\
& (\exists j\in L_{\sigma_{2}})(j\rightarrow t \mbox{ }\text{and} \mbox{ } L_{j-t}\cap L_{\sigma_{1}}=\emptyset\mbox{ }\text{for some} \mbox{ }j-t\mbox{ }\text{path})\}.
\end{align*}
Obviously $\Sigma_{1}\cap\Sigma_{2}=\emptyset$. Then we can construct two different solutions
$\textbf{x}^{1}, \textbf{x}^{2}\in\mathscr{X}^{\leq\textbf{x}}$ such that
$\textbf{x}=\textbf{x}^{1}\oplus\textbf{x}^{2}$.
We construct these solutions as follows. First we define the two possibly empty index sets
\begin{equation*}
\begin{split}
L= &\;\{l\in \text{Supp}(\textbf{x})\backslash(\Sigma_{1}\cup\Sigma_{2})\colon a_{li}+x_{i}>x_{l} \text{ for some } i\in \Sigma_{1}\\
 & \text{ and } a_{lj}+x_{j}<x_{l} \text{ for all } j\notin \Sigma_{1}\},\\
M= &\;\{m\in \text{Supp}(\textbf{x})\backslash(\Sigma_{1}\cup\Sigma_{2})\colon a_{mj}+x_{j}>x_{m} \text{ for some } j\in \Sigma_{2}\\ 
&\text{ and } a_{mi}+x_{i}<x_{m} \text{ for all } i\notin \Sigma_{2}\}.
\end{split}
\end{equation*}
If $L\neq \emptyset$, then for any $i\in L$ denote $L_{i}=\{j\in\Sigma_{1}: a_{ij}+x_{j}>x_{i}\}$ and let
\begin{equation*}
a=\min\limits_{i\in L}\max\limits_{j\in L_{i}}(a_{ij}+x_{j}-x_{i}).
\end{equation*}
If $M\neq \emptyset$, then for any $i\in M$ denote $M_{i}=\{j\in\Sigma_{2}: a_{ij}+x_{j}>x_{i}\}$ and let
\begin{equation*}
b=\min\limits_{i\in M}\max\limits_{j\in M_{i}}(a_{ij}+x_{j}-x_{i}).
\end{equation*}
With these notations, we define
$\textbf{x}^{1}=(x^{1}_1, x^{1}_2, \cdots, x^{1}_n)^T$ where $x^{1}_i=x_{i}-c$ for any $i\in \Sigma_{1}$ and $0<c\leq a$ if $L\neq\emptyset$ and $c>0$ if $L=\emptyset$, and  $x^{1}_j=x_{j}$ for any $j\notin \Sigma_{1}$.  We also define $\textbf{x}^{2}=(x^{2}_1, x^{2}_2, \cdots, x^{2}_n)^T$ where $x^{2}_j=x_{j}-d$ for any $j\in \Sigma_{2}$ and $0<d\leq b$ if $M\neq\emptyset$ and $d>0$ if $M=\emptyset$, and $x^{2}_i=x_{i}$ for any $i\notin \Sigma_{2}$.

Let us also make the following observation, to simplify our expressions for $L$ and $M$. Namely, observe that we cannot have
$a_{sj}+x_j=\max_{i\in N} a_{si}+x_i=x_s$ with $s\in  \text{Supp}(\textbf{x})\backslash(\Sigma_{1}\cup\Sigma_{2})$ and $j\in \Sigma_1\cup\Sigma_2$,
because then $s$ can be accessible from $\Sigma_1$ or $\Sigma_2$ (implying $s\in\Sigma_1$ or $s\in\Sigma_2$),
in contradiction with $s\in \text{Supp}(\textbf{x})\backslash(\Sigma_{1}\cup\Sigma_{2})$.
This observation allows us to write the following:
\begin{equation*}
\begin{split}
L&=\{i\in\text{Supp}(\textbf{x})\backslash(\Sigma_1\cup\Sigma_2)\colon a_{ij}+x_j<x_i\quad \forall j\notin\Sigma_1\}\\
M&=\{i\in\text{Supp}(\textbf{x})\backslash(\Sigma_1\cup\Sigma_2)\colon a_{ij}+x_j<x_i\quad \forall j\notin\Sigma_2\}
\end{split}
\end{equation*}

Obviously $\textbf{x}=\textbf{x}^{1}\oplus\textbf{x}^{2}$ and $\textbf{x}\neq\textbf{x}^{1}$, $\textbf{x}\neq\textbf{x}^{2}$. Hence it is sufficient to prove that $\textbf{x}^{1}\in \mathscr{X}$ and $\textbf{x}^{2}\in \mathscr{X}$.
To prove $\textbf{x}^{1}\in \mathscr{X}$ we will consider the following special cases:
\begin{itemize}
\item[(0)] If $i\notin \text{Supp}(\textbf{x})$, then  $A_{i}\otimes\textbf{x}^{1}\geq x_i^1=-\infty$ is obvious.
\item[(1)] If $i\in\Sigma_{1}$ then the construction of $\Sigma_{1}$ yields $(j,i)\in E_{\textbf{x}}$ for some $j\in\Sigma_{1}$. Then
$A_{i}\otimes\textbf{x}^{1}\geq a_{ij}+x_{j}^{1}=a_{ij}+x_{j}-c=x_{i}-c=x_{i}^{1}$, thus $A_{i}\otimes\textbf{x}^{1}\geq x_i^1$.
\item[(2)] If $i\in\Sigma_{2}$ then the construction of $\Sigma_{2}$ yields $(j,i)\in E_{\textbf{x}}$ for some $j\in\Sigma_{2}$. Then
$A_{i}\otimes\textbf{x}^{1}\geq a_{ij}+x_{j}^{1}=a_{ij}+x_{j}=x_{i}=x_{i}^{1}$, thus again $A_{i}\otimes\textbf{x}^{1}\geq x_i^1$.
\item[(3)] If $i\in\text{Supp}(\textbf{x})\backslash(\Sigma_{1}\cup\Sigma_{2})$ and $i\notin L$ then there exists $j\notin\Sigma_{1}$ such that $a_{ij}+x_{j}\geq x_{i}$. Then $A_{i}\otimes\textbf{x}^{1}\geq a_{ij}+x_{j}^{1}=a_{ij}+x_{j}\geq x_{i}=x_{i}^{1}$, thus $A_{i}\otimes\textbf{x}^{1}\geq x_i^1$.
\item[(4)] If $i\in\text{Supp}(\textbf{x})\backslash(\Sigma_{1}\cup\Sigma_{2})$ and $i\in L$ then, as $L\neq\emptyset$, we obtain:
\begin{align*}
&A_{i}\otimes\textbf{x}^{1}\geq\sum\limits_{j\in L_{i}}^{\oplus}(a_{ij}+x^{1}_j)\geq\sum\limits_{j\in L_{i}}^{\oplus}(a_{ij}+x_j-a)\\
&=\sum\limits_{j\in L_{i}}^{\oplus}(a_{ij}+x_j)-a\geq\sum\limits_{j\in L_{i}}^{\oplus}(a_{ij}+x_j)-\max\limits_{j\in L_{i}}(a_{ij}+x_{j}-x_{i})\\
&=\sum\limits_{j\in L_{i}}^{\oplus}(a_{ij}+x_j)-\sum\limits_{j\in L_{i}}^{\oplus}(a_{ij}+x_{j})+x_{i}\\
&=x_{i}=x_{i}^{1}
\end{align*}
thus $A_{i}\otimes\textbf{x}^{1}\geq x_i^1$.
\end{itemize}
Hence $\textbf{x}^{1}\in \mathscr{X}$, and the claim that $\textbf{x}^{2}\in \mathscr{X}$
can be proved similarly. 
\end{proof}

\begin{proposition}
\label{Pro1.5}
Let $\emph{\textbf{x}}=(x_1, x_2, \cdots, x_n)^T\in\mathscr{X}$. If $\emph{\text{Supp}}(\emph{\textbf{x}})$ does not have any proper subset that is isolated in $\mathcal{D}_{\emph{\textbf{x}}}$ and if all nodes in $\mathcal{D}_{\emph{\textbf{x}}}$ are invariable, then $\emph{\textbf{x}}'<\emph{\textbf{x}}$
for any $\emph{\textbf{x}}'\in\mathscr{X}^{\leq\emph{\textbf{x}}}$.
\end{proposition}
\begin{proof}
We first prove that
the conditions of the theorem hold if and only if $\mathcal{D}_{\textbf{x}}$ is an elementary cycle containing all nodes of $\text{Supp}(\textbf{x})$.

Set $|\text{Supp}(\textbf{x})|=s$. For any $i\in \text{Supp}(\textbf{x})$, since $i$ is an invariable node, there is $j_{i}\in \text{Supp}(\textbf{x})$ $(j_{i}\neq i)$ such that $(i, j_{i})\in E_{\textbf{x}}$ and $j_{i}$ has no other incoming arc. Then $s$ invariable nodes need at least $s$ different such nodes $j_{i}$. This implies that every node in $\text{Supp}(\textbf{x})$ has a unique incoming arc that is not a loop. Then $\mathcal{D}_{\textbf{x}}$ has precisely $s$ arcs which are not loops. As every invariable node has at least one outgoing arc which is not a loop, we obtain that every node in $\text{Supp}(\textbf{x})$ has unique outgoing arc and unique incoming arc that are not loops. This defines a permutation on $\text{Supp}(\textbf{x})$. In general, this permutation can be decomposed into a number of cycles, but we can have only one cycle since $\text{Supp}(\textbf{x})$ does not have any proper subset that is isolated in $\mathcal{D}_{\textbf{x}}$. Thus, if the conditions of the theorem hold, then $\mathcal{D}_{\textbf{x}}$ is an elementary cycle
containing all nodes of $\text{Supp}(\textbf{x})$. The converse statement is obvious.

Now let such elementary cycle in $\mathcal{D}_{\textbf{x}}$ be $\sigma=(i_{1},i_{2},\cdots, i_{s}, i_{s+1}=i_{1})$. Obviously, there are no arcs other than $(i_{p-1}, i_{p})$ for all $2\leq p\leq s+1$ in $\mathcal{D}_{\textbf{x}}$. By the definition of arc in $\mathcal{D}_{\textbf{x}}$, for any $i_{p}$ $(2\leq p\leq s+1)$, we have
\begin{eqnarray}\label{Eq1.1}
A_{i_{p}}\otimes\textbf{x}=\sum\limits_{1\leq t\leq n}^{\oplus}(a_{i_{p}t}+x_{t})=a_{i_{p}i_{p-1}}+x_{i_{p-1}}=x_{i_{p}}
\end{eqnarray}
and
\begin{eqnarray}\label{Eq1.2}
(a_{i_{p}t}+x_{t})<x_{i_{p}}\mbox{ }\text{ for any }\mbox{ } t\neq i_{p-1}.
\end{eqnarray}
Take any $\textbf{x}'=(x_1', x_2', \cdots, x_n')^T\in\mathscr{X}^{\leq\textbf{x}}$. We assume without loss of generality that  $x_{i_{1}}'<x_{i_{1}}$. Then by (\ref{Eq1.1}) and (\ref{Eq1.2}), inequality $A_{i_{2}}\otimes \textbf{x}' \geq x'_{i_2}$ implies that $x_{i_{2}}'$ should satisfy $x_{i_{2}}'<x_{i_{2}}$. Otherwise,
$A_{i_{2}}\otimes\textbf{x}'=(a_{i_{2}i_{1}}+x_{i_{1}}')\oplus\sum\limits_{t\neq i_{1}}^{\oplus}(a_{i_{2}t}+x_{t}')<(a _{i_{2}i_{1}}+x_{i_{1}})\oplus x_{i_{2}}=x_{i_{2}}=x_{i_{2}}'$, a contradiction. Repeating the same argument for $i_{3},i_{4},\cdots, i_{s}$,  we obtain $x_{i}'<x_{i}$ for all $i\in \text{Supp}(\textbf{x})$ and the statement follows.
\end{proof}

Proposition \ref{Pro1.5} shows that if $\text{Supp}(\textbf{x})$ does not have proper subset that is isolated in $\mathcal{D}_{\textbf{\textbf{x}}}$  and all nodes in $\mathcal{D}_{\textbf{x}}$ are invariable for some $\textbf{x}=(x_1, x_2, \cdots, x_n)^T$ in $\mathscr{X}$ then $\textbf{x}\neq\textbf{x}^{1}\oplus\textbf{x}^{2}$ for any $\textbf{x}^{1}, \textbf{x}^{2}\in\mathscr{X}^{\leq\textbf{x}}$, hence by Lemma \ref{Lem2} such solution $\textbf{x}$ is an extremal solution. We summarize:

\begin{corollary}
\label{c:important}
Let $\emph{\textbf{x}}=(x_1, x_2, \cdots, x_n)^T\in\mathscr{X}$. If $\emph{\text{Supp}}(\emph{\textbf{x}})$ does not have any proper subset that is isolated in $\mathcal{D}_{\emph{\textbf{x}}}$ and if all nodes in $\mathcal{D}_{\emph{\textbf{x}}}$ are invariable, then \emph{\textbf{x}} is an extremal.
\end{corollary}

\begin{proposition}
\label{Pro:main}
Let $\emph{\textbf{x}}=(x_1, x_2, \cdots, x_n)^T\in\mathscr{X}$ and suppose that $\mathcal{D}_{\emph{\textbf{x}}}$ satisfies the following conditions:
\begin{itemize}
\item[{\rm (i)}] $\emph{\text{Supp}}(\emph{\textbf{x}})$ does not have any proper subset that is isolated in $\mathcal{D}_{\emph{\textbf{x}}}$;
\item[{\rm (ii)}] $\mathcal{D}_{\emph{\textbf{x}}}$ does not have two different node-disjoint cycles,
\item[{\rm (iii)}] There is only one variable node, say $i$, in $\mathcal{D}_{\emph{\textbf{x}}}$ and all other nodes of $\mathcal{D}_{\emph{\textbf{x}}}$ are invariable.
\end{itemize}
Then $x_{i}'<x_{i}$ for any $\emph{\textbf{x}}'=(x_1', x_2', \cdots, x_n')^T\in\mathscr{X}^{\leq\emph{\textbf{x}}}$.
\end{proposition}
\begin{proof}
Set $|\text{Supp}(\textbf{x})|=s$. By contradiction, suppose that there is a solution $\textbf{x}'=(x_1', x_2', \cdots, x_n')^T\in\mathscr{X}^{\leq\textbf{x}}$ such that $x_{i}'=x_{i}$. Then
there exists $i_{1}\in \text{Supp}(\textbf{x})$, $i_{1}\neq i$ such that $x_{i_{1}}'<x_{i_{1}}$.

Since $i_{1}$ is invariable, there exists $i_{2}\in \text{Supp}(\textbf{x})$, $i_{2}\neq i_{1}$ such that $(i_{1}, i_{2})\in E_{\textbf{x}}$ and $i_{2}$ has no other incoming arc, that is,
$A_{i_{2}}\otimes\textbf{x}=\sum\limits_{1\leq j\leq n}^{\oplus}(a_{i_{2}j}+x_{j})=a_{i_{2}i_{1}}+x_{i_{1}}=x_{i_{2}}$
and
$(a_{i_{2}j}+x_{j})<x_{i_{2}}\mbox{ }\text{ for any }\mbox{ } j\neq i_{1}.$
Then $x_{i_{1}}'<x_{i_{1}}$ implies $x_{i_{2}}'<x_{i_{2}}$ since $\textbf{x}'=(x_1', x_2', \cdots, x_n')^T\in\mathscr{X}^{\leq\textbf{x}}$. If $i_{2}=i$, then we obtain a contradiction with $x_{i}'=x_{i}$.
If $i_{2}\neq i$, then we can find $(i_{2}, i_{3})\in E_{\textbf{x}}$ and  $x_{i_{2}}'<x_{i_{2}}$ implies $x_{i_{3}}'<x_{i_{3}}$  can be similarly obtained, since $i_{2}$ is also invariable. Continuing like this, we obtain an elementary path
\begin{eqnarray}\label{Eq1.4}
\pi=(i_{1}, i_{2}, \cdots, i_{p}, i_{p+1}=i),
\end{eqnarray}
where $l(\pi)\geq 2$ and $i_{h}$ has unique incoming arc $(i_{h-1}, i_{h})$ for any $2\leq h\leq p+1$,
or an elementary cycle
\begin{eqnarray}\label{Eq1.5}
\sigma=(j_{1}=i_{1}, j_{2}=i_{2}, \cdots, j_{q}, j_{q+1}=j_{1}),
\end{eqnarray}
where $l(\sigma)\geq 2$ and
$j_{t}$ has unique incoming arc $(j_{t-1}, j_{t})$ for any $2\leq t\leq q+1$ and $i\notin L_\sigma$.

In the first case, we can obtain successively: $$x_{i_{1}}'<x_{i_{1}}, \cdots, x_{i_{p}}'<x_{i_{p}}, x_{i}'<x_{i},$$ a contradiction with $x_{i}'=x_{i}$. In the second case, we can get successively:
\begin{eqnarray}\label{Eq1.3}
x_{j_{1}}'<x_{j_{1}}, x_{j_{2}}'<x_{j_{2}}, \cdots, x_{j_{q}}'<x_{j_{q}}.
\end{eqnarray}

Furthermore, if $q=s-1$, then by the conditions of the proposition, $i$ has no loop ($(i, i)\in E_{\textbf{x}}$ contradicts condition (ii)) and has at least one incoming arc from node in $L_\sigma$, that is $(j_{r}, i)\in E_{\textbf{x}}$ for some $1\leq r\leq q$.
Note that no matter how many of such $j_{r}$ we have, \eqref{Eq1.3} must lead to $x_{i}'<x_{i}$, a contradiction with $x_{i}'=x_{i}$.

We now assume that $q\leq s-2$. Set $\{j_{q+1},j_{q+2},\ldots, j_{s-1}\}=\text{Supp}(\textbf{x})\backslash (L_{\sigma}\cup\{i\})$.
We argue that for each node in this set we can find a path to $i$ such that every arc $(j,j')$ in this path has the property that $j'$ has no other incoming arc in the tangent digraph. Without loss of generality, take $j_{q+1}$. Since this is an invariable node, such arc $(j,j')$ exists for $j=j_{q+1}$, and either it is $(j_{q+1},i)$ or it is (without loss of generality) $(j_{q+1},j_{q+2})$, and then we can continue with $j_{q+2}$. Note that this process will build a path to $i$ in a finite number of steps, since for every new node $j$ we see that: 1) an arc $(j,j')$ with the above mentioned property always exists, 2) $j'\notin\{j_1,\ldots,j_q\}$, and  $j'$ is not among the nodes on the path from $j_{q+1}$ to $j$ since there cannot be any cycle other than the one that we found above, 3) the number of the nodes is finite and we will have $j'=i$ at some stage.

Thus we obtain a directed tree spanning the nodes  $\{j_{q+1},j_{q+2},\ldots, j_{s-1}, i\}$, whose only root is $i$. Furthermore, this tree is a directed path to $i$, since each arc of the tree has the property that its end node has no other incoming arcs. Without loss of generality, let $j_{q+1}$ be the starting node of that path, and the path be $(j_{q+1},j_{q+2},\ldots,j_{s-1},i)$. Then $j_{q+1}$ is the only node to which we have arcs from $\{j_1,\ldots, j_q\}$. Note that we do not have any arcs ending at $j_{q+1}$ and beginning at the nodes of $\{j_{q+1},j_{q+2},\ldots,j_{s-1},i\}$, as this would create a cycle. Also, no matter how many such arcs from $\{j_1,\ldots, j_q\}$ to $j_{q+1}$ we have, \eqref{Eq1.3} must lead to $x_{j_{q+1}}'<x_{j_{q+1}}$. Following the path $(j_{q+1},j_{q+2},\ldots,j_{s-1},i)$ we then have
$$
x'_{j_{q+2}}<x_{j_{q+2}},\ldots, x'_{j_{s-1}}<x_{j_{s-1}}, x'_i<x_i,
$$
in contradiction with $x'_i=x_i$.
\end{proof}

\begin{remark}
\label{r:proofcase}
Contemplating the proof given above for the case when $\mathcal{D}_{\emph{\textbf{x}}}$ contains a cycle $\sigma=(j_{1}, j_{2}, \cdots, j_{q}, j_{q+1}=j_{1})$, where $l(\sigma)\geq 2$ and
$j_{t}$ has unique incoming arc $(j_{t-1}, j_{t})$ for any $2\leq t\leq q+1$, we see that in this case the variable node $i$ cannot have any outgoing arcs, because this would contradict either the properties of $\{j_1,\ldots, j_q\}$ or that $\mathcal{D}_{\emph{\textbf{x}}}$ does not contain two disjoint cycles. Therefore, in this case $i$ can be only I-variable node with no loop on it, but not a II-variable node.
\end{remark}

The next result can be deduced now from Proposition~\ref{Pro:main}:

\begin{corollary}
Let $\emph{\textbf{x}}=(x_1, x_2, \cdots, x_n)^T\in\mathscr{X}$. If $\mathcal{D}_{\emph{\textbf{x}}}$ satisfies the following conditions:
\begin{itemize}
\item[{\rm (i)}] $\emph{\text{Supp}}(\emph{\textbf{x}})$ does not have proper subsets that are isolated in
$\mathcal{D}_{\emph{\textbf{x}}}$,
\item[{\rm (ii)}] $\mathcal{D}_{\emph{\textbf{x}}}$ does not have two different node-disjoint cycles,
\item[{\rm (iii)}] there is only one variable node in $\mathcal{D}_{\emph{\textbf{x}}}$ and all other nodes are invariable.
\end{itemize}
Then \emph{\textbf{x}} is an extremal.
\end{corollary}
\begin{proof}
It follows from Proposition~\ref{Pro:main} that $x_{i}'<x_{i}$ for any $\textbf{x}'\in\mathscr{X}^{\leq}$, where $i$ is the unique variable node in $\mathcal{D}_{\textbf{x}}$. Then $\textbf{x}\neq\textbf{x}^{1}\oplus\textbf{x}^{2}$ for any $\textbf{x}^{1}, \textbf{x}^{2}\in\mathscr{X}^{\leq\textbf{x}}$ and the statement follows.
\end{proof}

The results of this section can be summarized in the following characterization of extremals.

\begin{theorem}
\label{Th:main}
Let $\emph{\textbf{x}}=(x_1, x_2, \cdots, x_n)^T\in\mathscr{X}$. Then $\emph{\textbf{x}}$ is an extremal if and only if $D_{\emph{\textbf{x}}}$ satisfies the following conditions:
\begin{itemize}
\item[{\rm (i)}] $\emph{\text{Supp}}(\emph{\textbf{x}})$ does not have any proper subset isolated in
$D_{\emph{\textbf{x}}}$,
\item[{\rm (ii)}] $\mathcal{D}_{\emph{\textbf{x}}}$ does not have two different node-disjoint cycles,
\item[{\rm (iii)}] there is at most one variable node in $\mathcal{D}_{\emph{\textbf{x}}}$.
\end{itemize}
\end{theorem}

Let us illustrate our main result by the following example.

\begin{example}
\label{Ex3.1}
{\rm 
Consider the system $A\otimes \textbf{x}\geq\textbf{x}$, where
$$A=\left(\begin{array}{ccccc}
                   -5 & 0 & -\infty & -\infty & -\infty \\
                   0 & -\infty & -\infty & -\infty & -\infty \\
                   0 & -\infty & -\infty & -\infty & -\infty \\
                   -\infty & -\infty & -3 & -\infty & 0 \\
                   -\infty & -\infty & -\infty & 0 & -\infty
                 \end{array}
    \right)
.$$
For solution $\textbf{x}^{1}=(0, 0, 0, -3, -\infty)^{T}$, the tangent digraph $\mathcal{D}_{\textbf{x}^{1}}=(\text{Supp}(\textbf{x}^{1}), E_{\textbf{x}^{1}})$ at $\textbf{x}^{1}$ is given on Figure \ref{f1}.
It is easy to check that $\text{Supp}(\textbf{x}^{1})$ does not have proper isolated subsets and $\mathcal{D}_{\textbf{x}^{1}}$ does not have two different node-disjoint cycles and there is only one variable node $4$ in $\mathcal{D}_{\textbf{x}^{1}}$ (the node $4$ is the \uppercase\expandafter{\romannumeral1}-variable node). Hence by Theorem 3.1 $\textbf{x}^{1}=(0, 0, 0, -3, -\infty)^{T}$ is an extremal.


For another solution $\textbf{x}^{2}=(0, 0, 0, 0, 0)^{T}$, the tangent digraph $\mathcal{D}_{\textbf{x}^{2}}=(\text{Supp}(\textbf{x}^{2}), E_{\textbf{x}^{2}})$ at $\textbf{x}^{2}$ is given on Figure \ref{f2}.
Obviously, there are two isolated node sets in $\mathcal{D}_{\textbf{x}^{2}}$, that is, $\{1, 2, 3\}$ and $\{4, 5\}$. Then by Theorem 3.1 we get that $\textbf{x}^{2}=(0, 0, 0, 0, 0)^{T}$ is not an extremal. In fact, $\textbf{x}^{2}=\textbf{x}^{2'}\oplus\textbf{x}^{2''}$, where $\textbf{x}^{2'}=(0, 0, 0, -\infty, -\infty)^{T}$, $\textbf{x}^{2''}=(-\infty, -\infty, -\infty, 0, 0)^{T}$ and $\textbf{x}^{2'}, \textbf{x}^{2''}$ are both solutions satisfying $\textbf{x}^{2}\neq\textbf{x}^{2'}$ and $\textbf{x}^{2}\neq\textbf{x}^{2''}$.}
\end{example}

\begin{figure}[H]
\begin{center}
    \begin{tikzpicture}[scale=0.8]
    \node[shape=circle,draw=black]  (1) at (0,0) {$2$};
    \node[shape=circle,draw=black]  (2) at (2.5,0) {$1$};
    \node[shape=circle,draw=black]  (3) at (5,0) {$3$};
    \node[shape=circle,draw=black]  (4) at (7.5,0) {$4$};

    \path [->](1) edge[bend left = 45] (2);
    \path [->](2) edge[bend left = 45] (1);
    \path [->](2) edge[ left = 10]  (3);
    \path [->](3) edge[ left = 10]  (4);
    \end{tikzpicture}
\caption{The tangent digraph $\mathcal{D}_{\textbf{x}^{1}}$ at $\textbf{x}^{1}$}
\label{f1}
\end{center}
\end{figure}

\begin{figure}[H]
\begin{center}
    \begin{tikzpicture}[scale=0.8]
    \node[shape=circle,draw=black]  (1) at (0,0) {$2$};
    \node[shape=circle,draw=black]  (2) at (2.5,0) {$1$};
    \node[shape=circle,draw=black]  (3) at (5,0) {$3$};
    \node[shape=circle,draw=black]  (4) at (7,0) {$4$};
    \node[shape=circle,draw=black]  (5) at (9.5,0) {$5$};

    \path [->](1) edge[bend left = 45] (2);
    \path [->](2) edge[bend left = 45] (1);
    \path [->](2) edge[ left = 10]  (3);
    \path [->](4) edge[bend left = 45] (5);
    \path [->](5) edge[bend left = 45] (4);
    \end{tikzpicture}
\caption{The tangent digraph $\mathcal{D}_{\textbf{x}^{2}}$ at $\textbf{x}^{2}$}
\label{f2}
\end{center}
\end{figure}

\if{
\begin{example}\hspace{-0.4em}\textbf{.}\label{Ex3.1}
\emph{
Consider the system $A\otimes \textbf{x}\geq\textbf{x}$, where
$$A=\left(
      \begin{array}{ccc}
        1 & 1 & -\infty \\
        2 & -2 & -\infty \\
        -\infty & 3 & -3 \\
      \end{array}
    \right)
.$$}

For solution $\emph{\textbf{x}}^{1}=(0, 2, 5)^{T}$, the tangent digraph $\mathcal{D}_{\emph{\textbf{x}}^{1}}=(\emph{\text{Supp}}(\emph{\textbf{x}}^{1}), E_{\emph{\textbf{x}}^{1}})$ at $\emph{\textbf{x}}^{1}$ is given on Figure \ref{f1}.
It is easy to check that $\emph{\text{Supp}}(\emph{\textbf{x}}^{1})$ does not have proper isolated subsets and $\mathcal{D}_{\emph{\textbf{x}}^{1}}$ does not have two different node-disjoint cycles (containing loop) and there is only one variable node $3$ in $\mathcal{D}_{\emph{\textbf{x}}^{1}}$ (the node $3$ is the \uppercase\expandafter{\romannumeral1}-variable node). Hence by Theorem \ref{Th:main} $\emph{\textbf{x}}^{1}=(0, 2, 5)^{T}$ is an extremal.


For another solution $\emph{\textbf{x}}^{2}=(0, -1, 2)^{T}$, the tangent digraph $\mathcal{D}_{\emph{\textbf{x}}^{2}}=(\emph{\text{Supp}}(\emph{\textbf{x}}^{2}), E_{\emph{\textbf{x}}^{2}})$ at $\emph{\textbf{x}}^{2}$ is given on Figure \ref{f2}.
Obviously, there are two isolated node sets in $\mathcal{D}_{\emph{\textbf{x}}^{2}}$, that is, $\{1\}$ and $\{2, 3\}$. Then by Theorem \ref{Th:main}  we get that $\emph{\textbf{x}}^{2}=(0, -1, 2)^{T}$ is not an extremal. In fact, $\emph{\textbf{x}}^{2}=\emph{\textbf{x}}^{2'}\oplus\emph{\textbf{x}}^{2''}$, where $\emph{\textbf{x}}^{2'}=(-3, -1, 2)$, $\emph{\textbf{x}}^{2''}=(0,-1, -\infty)$ and $\emph{\textbf{x}}^{2'}, \emph{\textbf{x}}^{2''}$ are both solutions satisfying $\emph{\textbf{x}}^{2}\neq\emph{\textbf{x}}^{2'}$ and $\emph{\textbf{x}}^{2}\neq\emph{\textbf{x}}^{2''}$.
\end{example}
\begin{figure}[H]
\begin{center}
    \begin{tikzpicture}[scale=0.8]
    \node[shape=circle,draw=black]  (1) at (0,0) {$1$};
    \node[shape=circle,draw=black]  (2) at (3,0) {$2$};
    \node[shape=circle,draw=black]  (3) at (6,0) {$3$};

    \path [->](1) edge[ left = 10]  (2);
    \path [->](2) edge[ left = 10]  (3);
    \end{tikzpicture}
\caption{The tangent digraph $\mathcal{D}_{\textbf{x}^{1}}$ at $\textbf{x}^{1}$}
\label{f1}
\end{center}
\end{figure}

\begin{figure}[H]
\begin{center}
    \begin{tikzpicture}[scale=0.8]
    \node[shape=circle,draw=black]  (1) at (0,0) {$1$};
    \node[shape=circle,draw=black]  (2) at (3,0) {$2$};
    \node[shape=circle,draw=black]  (3) at (6,0) {$3$};

    \path [->](2) edge[ left = 10]  (3);
    \end{tikzpicture}
\caption{The tangent digraph $\mathcal{D}_{\textbf{x}^{2}}$ at $\textbf{x}^{2}$}
\label{f2}
\end{center}
\end{figure}
}\fi

\section{Computational complexity}
\label{s:comp}

Theorem \ref{Th:main} provides a method for checking whether a given solution $\textbf{x}$ is an extremal or not: one has to check if all conditions (i)-(iii) hold. We are going to show that this can be done in linear ($O(|\text{Supp}(\textbf{x})|+|E_{\textbf{x}}|)$) time for any tangent digraph $D_{\textbf{x}}=(\text{Supp}(\textbf{x}), E_{\textbf{x}})$. Let us start with the following remarks:
\begin{itemize}
\item[(i)] The first condition of Theorem~\ref{Th:main} can be checked in $O(|\text{Supp}(\textbf{x})|+|E_{\textbf{x}}|)$ time by means of the depth-first search techique of~\cite{Tarjan1971}, for arbitrary digraph;
\item[(ii)] The second condition can be checked in polynomial time for arbitrary digraph, by the result of McCuaig~\cite{Mccuaig1993};
\item[(iii)] 
For each node we check the outgoing arcs: whether they exist and whether one of the outgoing arcs leads to a node with just one incoming arc. This can be done in $O(|E_{\textbf{x}}|)$ time.


\item[(iv)] Assume that condition (iii) holds and, moreover, all nodes are invariable. Then, by Corollary~\ref{c:important}, condition (i) of Theorem~\ref{Th:main} is sufficient for $\textbf{x}$ to be an extremal.
\end{itemize}
By remark (iv), we only need to consider the situation when condition (iii) of Theorem~\ref{Th:main} holds and exactly one of
 the nodes is variable and all other nodes are invariable. In this case, the result of \cite{Mccuaig1993} is not sufficient for linear complexity, demonstration of which will be based on 
 Algorithm~\ref{alg:extremality} stated below.

The following observation is implicit in the proofs written above:

\begin{lemma}
Suppose that a tangent digraph $D_{\emph{\textbf{x}}}=(\text{Supp}(\emph{\textbf{x}}), E_{\emph{\textbf{x}}})$, where\\
 $|\text{Supp}(\emph{\textbf{x}})|=s$, has exactly one variable node and all other $s-1$ are invariable.
Then at least $(s-1)$ nodes have exactly one incoming arc and no loop on them.
\end{lemma}
\begin{proof}
Each invariable node has at least one outgoing arc leading to a node with just one incoming arc and no loop on it.  For different invariable nodes such nodes with one incoming arc and no loop on them have to be different (otherwise the incoming arc is non-unique). Since there are $s-1$ invariable nodes, the number of nodes with exactly one incoming arc and no loop on them is at least $(s-1)$.
\end{proof}

If all nodes in a general digraph $\mathcal{D}=(V,E)$ have exactly one incoming arc, then $\mathcal{D}$ consists in general of a number of isolated components. Each component has just one cycle and a number of paths emanating from it, as it follows, e.g., from \cite[Section 3.4]{MPatWork}. In this case we see that conditions (i) and (ii) of Theorem~\ref{Th:main} are equivalent.

If we have $(|V|-1)$ nodes with exactly one incoming arc and no loops on them and one node with no incoming arc in a general digraph $\mathcal{D}=(V,E)$, then it gives rise to a component with no cycle. Therefore, in this case condition (i) implies condition (ii).

By the above discussion, when condition (iii) holds and exactly one of the nodes is variable and all other nodes are invariable, if all nodes have exactly one incoming arc, or if we have $(|V|-1)$ nodes with exactly one incoming arc and no loops on them and one node with no incoming arc, then checking condition (i) is sufficient for checking extremality. We now analyse the remaining case,
 where we have $(|V|-1)$ nodes with exactly one incoming arc and no loops on them, and one node with more than one incoming arc in general digraph $\mathcal{D}=(V,E)$. As we need to verify that there are no more than two disjoint cycles, we can equivalently consider the following kind of digraphs:
\begin{definition}
Digraph $\mathcal{D}=(V,E)$ with $|V|=n$ is called an {\em $o$-fountain}, if it has 1) $n-1$ nodes with only one outgoing arc and no loops on them, 2) one node $o$ with at least two outgoing arcs (one of which can be a loop).
\end{definition}

As all nodes but $o$ have only one outgoing arc, each arc outgoing from $o$ gives rise to a {\em jet:} a path ending with an elementary cycle. There are two kinds of jet: 1) {\em $\sigma$-jet:} a path starting at $o$ and ending with an elementary cycle $\sigma$ that does not contain $o$, 2) {\em $o$-cycle:} an elementary cycle containing $o$ (which could be also a loop on $o$).  Let us make a couple of observations.

\begin{lemma}
\label{l:observations}
The following properties hold for any $o$-fountain:
\begin{itemize}
\item[{\rm (i)}] For any $\sigma_1$-jet and $\sigma_2$-jet, either $\sigma_1=\sigma_2$ or $\sigma_1$ and $\sigma_2$ are node-disjoint;
\item[{\rm (ii)}] For any $\sigma$-jet, the cycle $\sigma$ is disjoint from any $o$-cycle.
\end{itemize}
\end{lemma}
\begin{proof}
For any $\sigma_1$-jet and $\sigma_2$-jet, either $\sigma_2$-jet intersects with $\sigma_1$-jet at a node other than $o$ and then it follows $\sigma_1$-jet, ending with the same elementary cycle $\sigma_1=\sigma_2$, or $o$ is the only common node of both jets and then $\sigma_1$ and $\sigma_2$ are node-disjoint.

Also, if we assume that $\sigma$-jet contains a node of $o$-cycle other than $o$, then we obtain a contradiction, since after that node the $o$-cycle follows the $\sigma$-jet and ends up with the same cycle $\sigma$ instead of returning to $o$.
\end{proof}

Let us also introduce the following notation:

\begin{definition}
The subgraph of $\mathcal{D}$ consisting of all nodes that are accessible from $o$ and all arcs outgoing from $o$ and from these nodes, is denoted by $\mathcal{D}_o$. Furthermore, denote by $\overline{\mathcal{D}}_o$ the component of $\mathcal{D}$ which consists of all nodes and arcs of $\mathcal{D}_o$, all nodes of $\mathcal{D}$ which have access to a node of $\mathcal{D}_o$ and all arcs outgoing from such nodes. Denote by $\mathcal{D}\backslash \overline{\mathcal{D}}_o$ the subgraph of $\mathcal{D}$ consisting of all nodes and arcs which are not in $\overline{\mathcal{D}}_o$.
\end{definition}

Obviously, $\overline{\mathcal{D}}_o$ has at least one cycle, since each jet emanating from $o$ ends with a cycle, and all of its cycles are in $\mathcal{D}_o$. Also, $\mathcal{D}\backslash \overline{\mathcal{D}}_o$ can be decomposed into a number of isolated components, each of which contains exactly one cycle (following \cite[Section 3.4]{MPatWork}) and, in particular, if $\mathcal{D}\backslash \overline{\mathcal{D}}_o$ is non-empty then it has a cycle. This, together with Lemma~\ref{l:observations}, and the discussions before that Lemma, shows the validity of the following algorithm,
where by $\mathcal{D}^-$ we denote the digraph obtained from $\mathcal{D}$ by swapping the directions of all arcs.

\begin{breakablealgorithm}
 \renewcommand{\algorithmicrequire}{\textbf{Input:}}
 \renewcommand{\algorithmicensure}{\textbf{Output:}}
 \caption{ \label{alg:extremality}}
 \begin{algorithmic}[1]
 \REQUIRE  Vector $\textbf{x}\in\mathscr{X}$.
 \STATE Construct the graphs $\mathcal{D}_{\textbf{x}}$ and $\mathcal{D}^-_{\textbf{x}}$.
 \IF {more than one node of $\mathcal{D}_{\textbf{x}}$ is variable}
 \STATE return {\bf false.}
 \ENDIF
 \IF {all nodes have no more than one incoming arc}
 \IF{Supp({\bf x}) has an isolated subset in $\mathcal{D}_{\textbf{x}}$}
 \STATE return {\bf false}
 \ELSE
 \STATE return {\bf true}
 \ENDIF
 \ENDIF
 \STATE Find a node $o$ with $t\geq 2$ outgoing arcs (one of which can be a loop) in  $\mathcal{D}^-_{\textbf{x}}$.
 \STATE Inspect all $t$ jets emanating from $o$ in $\mathcal{D}^-_{\textbf{x}}$.
  \IF{there is an $o$-cycle and a $\sigma$-jet {\bf OR} there are $\sigma_1$-jet and $\sigma_2$-jet with $\sigma_1\neq\sigma_2$}
 \STATE return {\bf false}
 \ENDIF
 \STATE Construct the subgraph $\overline{\mathcal{D}}_o$ of $\mathcal{D}:=\mathcal{D}^-_{\textbf{x}}$.
 \IF{$\mathcal{D}\backslash \overline{\mathcal{D}}_o$ is non-empty}
 \STATE return {\bf false}
 \ELSE
 \STATE return {\bf true}
 \ENDIF
 \ENSURE Whether or not $\mathbf{x}$ is an extremal.
 \end{algorithmic}
\end{breakablealgorithm}

\begin{remark}
\label{r:DD-}
In the above algorithm, we construct graph $\mathcal{D}^-_{\emph{\textbf{x}}}$ only for the convenience of formulation.
The actual computations can be performed solely in terms of $\mathcal{D}_{\emph{\textbf{x}}}$.
\end{remark}

\begin{theorem}
Checking the extremality of $\emph{\textbf{x}}\in\mathscr{X}$ by Algorithm~\ref{alg:extremality}
requires no more than $O(n^2)$ operations.
\end{theorem}
\begin{proof}
In the beginning of the algorithm, checking if more than one node of $\mathcal{D}_{\textbf{x}}$ is variable or if Supp({\bf x}) has an isolated subset in $\mathcal{D}_{\textbf{x}}$ requires no more than $O(n^2)$ by the above discussion. Next we inspect the jets emanating from $o$ to search for disjoint cycles, and in this inspection we do not need to inspect any arc more than once, therefore the complexity at this step is proportional to the number of arcs in $\mathcal{D}_o$. The complexity of the rest of the algorithm is proportional to the number of arcs in $\mathcal{D}^-_{\textbf{x}}\backslash\mathcal{D}_o$. The claim now follows by adding up the complexities of all parts of the algorithm. \end{proof}

\end{document}